\documentclass[english]{smfart}

  \usepackage{macros}

\author{Beno\^{\i}t Kloeckner}
\title{The space of closed subgroups of $\bR^n$}

\newcommand{\chab}{\mathscr{C}}

\newcommand{\dotprod}{\cdot}

\begin{document}

\begin{abstract}
The Chabauty space of a topological group is the 
set of its closed subgroups, endowed with a natural topology. As soon as $n>2$,
the Chabauty space of $\bR^n$ has a rather intricate topology and is not a manifold.
By an investigation of its local structure, 
we fit it into a wider, but not too wild, class
of topological spaces (namely Goresky-MacPherson stratified spaces).
Thanks to a localization theorem, this local study also
leads to the main result of this article: the Chabauty space of $\bR^n$ 
is simply connected for all $n$.
Last, we give an alternative proof of the Hubbard-Pourezza Theorem, which
describes the Chabauty space of $\bR^2$.
\end{abstract}

\maketitle

\section{Introduction}

Let $G$ be a topological group whose neutral element is denoted by $0$
(although $G$ need not be abelian). 
Its \emph{Chabauty space} $\chab(G)$ is the set
of closed subgroups of $G$ endowed with the following topology: the neighborhoods
of a point $\Gamma\in\chab(G)$ are the sets
$$\mathcal{N}_U^K(\Gamma)=\ensemble{\Gamma'\in G}{\Gamma'\cap K\subset \Gamma\cdot U\mbox{ and }
\Gamma\cap K\subset \Gamma'\cdot U}$$
where $K$ runs over the compact subsets of $G$ and $U$ runs over the neighborhoods of $0$.
In words, $\Gamma'$ is very close to $\Gamma$ if, on a large compact set, every of its elements
is in a uniformly small neighborhood of an element of $\Gamma$, and conversely. 
The preprint \cite{delaHarpe} contains a more detailed account of this topology.

The Chabauty space is named after Claude Chabauty, who introduced it in \cite{Chabauty}
to generalize Mahler's compactness criterion
to lattices in locally compact groups. If $G$ is locally compact, then $\chab(G)$
is compact and can therefore be used to define a compactification
of any space whose points are naturally associated to closed subgroups of $G$.
For example, this is the case of a symmetric space of noncompact type: one 
maps a point to its stabilizer in the isometry group. The corresponding
compactification is isomorphic to the Satake compactification
\cite{Satake,Borel-Ji}. This compactification was generalized to buildings
thanks to the Chabauty topology point of vue in \cite{Guivarch-Remy}.

The simplest example of a Chabauty space is that of the line: $\chab(\bR)$ contains the trivial subgroup
$\{0\}$, the discrete groups $\alpha\bZ$ and the total group $\bR$. Two
discrete groups $\alpha\bZ$ and $\beta\bZ$ are close one to another when
$\alpha$ and $\beta$ are close, a neighborhood of $\{0\}$ consists in
the set of $\alpha\bZ$ with large $\alpha$ (and we define
$\infty\bZ=\{0\}$) and a neighborhood of $\bR$ consists in the set of
$\alpha\bZ$ with small $\alpha$ (and we define $0\bZ=\bR$). Putting all this together,
we see that $\chab(\bR)$ is homeomorphic to a closed interval.

\begin{figure}[htp]\begin{center}
\input{chab_line.pstex_t}
\caption{Chabauty space of $\bR$.}
\end{center}\end{figure}

Only for a few groups $G$ is known a precise description of $\chab(G)$.
Recent work of Bridson, de la Harpe and Kleptsyn adds to the list
the three-dimensional Heisenberg group \cite{BHK}, but the topology of 
$\chab(\bR^n)$ is \textit{terra incognita} for $n>2$.
Even $\chab(\bR^2)$ is uneasy to describe; it was tackled by Hubbard and Pourezza \cite{Hubbard-Pourezza}
who proved the following. 

\begin{theoalph}[Hubbard-Pourezza]\label{theo:HP}
Let $\chab$ be the Chabauty space of $\bR^2$ and $\mathscr{L}$ be the subset of lattices.
The topological pair $(\chab,\chab\smallsetminus\mathscr{L})$ is homeomorphic to the suspension of
$(S^3,K)$ where $K$ is a trefoil knot in the $3$-sphere. In particular,
$\chab$ is a $4$-sphere.
\end{theoalph}

Let us recall some definitions. A \emph{topological pair} is a pair $(X,Y)$ of
topological spaces where $Y$ is a subset of $X$ (endowed with the induced topology).
Two topological pairs $(X,Y)$ and $(X',Y')$ are \emph{homeomorphic} if
there is a homeomorphism $\Phi: X\to X'$ that maps $Y$ onto $Y'$. The \emph{(open) cone}
over $X$ is the quotient
$cX$ of $X\times [0,1)$ by the relation $(x_0,0)\sim(x_1,0)$, while
the \emph{suspension}
of $X$ is the quotient $sX$ of $X\times [0,1]$ by the relations
$(x_0,0)\sim(x_1,0)$ and $(x_0,1)\sim(x_1,1)$ for all $x_0,x_1\in X$.
If $Y$ is a subset of $X$, then $sY$ embeds naturally in $sX$
and the resulting topological pair $(sX,sY)$ is called
the suspension of $(X,Y)$. 

The Hubbard-Pourezza theorem shows in particular that the set of non-lattices is a $2$-sphere
that is \emph{non-tamely} embedded in $\chab(\bR^2)$. At the end of the paper
we shall give a proof of Theorem \ref{theo:HP} using Seifert fibrations. It is likely to be a variation of the alternative
proof alluded to in \cite{Hubbard-Pourezza}, but we provide it so that this article is self-contained;
we shall indeed see that Theorem \ref{theo:HP} gives information on the ``links'' of some points
in $\chab(\bR^n)$ for higher $n$.

Our main goal is to investigate the space $\chab(\bR^n)$. It is not
a manifold when $n>2$, and our first aim is to show that
it fits into the more general framework of stratified spaces.

\begin{theoalph}\label{theo:stratif}
For all $n$, the Chabauty space of $\bR^n$ admits a Goresky-MacPherson
stratification. If $n\geqslant 2$ it is moreover a pseudo-manifold.
\end{theoalph}

We shall give  Goresky and MacPherson's definitions of a stratification 
and pseudo-manifold later on; roughly, it means that $\chab(\bR^n)$ is a union of
manifolds nicely glued together. Compact Goresky-MacPherson
stratified spaces have for example a well-defined intersection homology,
locally contractible homeomorphism group and extension of isotopy properties.

To prove Theorem \ref{theo:stratif} we shall unveil part of the
\emph{local} topology of $\chab(\bR^n)$.
Its global topology seems difficult to make completely explicit.
It might be possible to describe precisely $\chab(\bR^3)$,
but as $n$ grows the space becomes more and more complicated. Even if one
describes each of its strata, the way they glue together is quite involved
(even if locally trivial in some sense). In such a case, one tries
to compute some topological invariants to get a grip on the space,
the first one being its fundamental group. Our main result is the following.

\begin{theoalph}\label{theo:main}
For all $n$, the Chabauty space of $\bR^n$ is simply connected.
\end{theoalph}

Let us give a sketch of the proof of this result; in the sequel we often
use the notation $\chab:=\chab(\bR^n)$.
The subset $\mathscr{R}_m\subset \mathscr{C}$ of maximal rank subgroups (that is, subgroups
containing a basis of vectors of $\bR^n$) is open, dense and contractible.
Its complement $\mathscr{R}_\ell:=\chab\smallsetminus\mathscr{R}_m$ is a subspace of codimension $n$. If $\chab$
 where a differentiable
manifold, we could have proceeded by tansversality arguments:
any loop based in a point of $\mathscr{R}_m$ would be homotopic to a generic smooth loop,
transversal to $\mathscr{R}_\ell$. But transversality between a submanifold
of codimension $>1$ and a curve would imply that they do not meet,
and since $\mathscr{R}_m$ is contractible the loop would then be nullhomotopic.

One must be very carefull when trying to apply these
arguments in more general spaces. For example, the cone over
a disconnected manifold is a stratified space,
 and it is not true that a generic curve avoids
its apex.
\begin{figure}[htp]\begin{center}
\includegraphics[scale=.8]{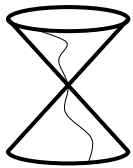}
\caption{A generic curve does not avoid the apex.}
\end{center}\end{figure}

This example is however very local in nature, and one guesses that
if no such phenomenon occurs, then one should be able to
proceed almost as if $\chab$ where a manifold. This guess is true,
as is shown by the following relative homotopy localization result.

\begin{theoalph}[localization]\label{theo:localization}
Let $X$ be a Hausdorff topological space
and $Y$ be a closed subset of $X$.

If every point $y\in Y$ admits a neighborhood system $(U_\varepsilon)_\varepsilon$
such that each pair $(U_\varepsilon,U_\varepsilon\smallsetminus Y)$ is $k$-connected,
then the pair $(X,X\smallsetminus Y)$ is $k$-connected.
\end{theoalph}

In other words, under very mild assumptions a pair that is locally
$k$-connected must be globally $k$-connected. Let us recall
that a pair $(X,U)$ is $0$-connected if any point in $X$ can be connected
by a continuous path to a point in $U$. The pair is $k$-connected if moreover
for all $\ell\leqslant k$, every map $(I^\ell,\partial I^\ell)\to (X,U)$
from the closed cube $I^\ell$ that maps its boundary into $U$ is homotopic 
(relative to its boundary) to a map $I^\ell\to U$.

It would be very surprising that such a simple and helpful result be new, however
I could not find a reference in the litterature (except
\cite{Eyral} where it is proved only for polyhedral pairs) and we shall therefore
provide a proof.

The topics of the next sections are: some preliminaries and definitions (Section \ref{sec:def}),
 stratifications (Section \ref{sec:strat}),
the local study of $\chab(\bR^n)$ --including the proof of Theorem \ref{theo:stratif}
(Section \ref{sec:local}), proofs
of Theorems \ref{theo:main} and \ref{theo:localization} (Section \ref{sec:global}),
alternative proof of Theorem \ref{theo:HP} (Section \ref{sec:plane}), and some open
questions (Section \ref{sec:open}).

\paragraph*{Aknowledgements}
I wish to thank Lucien Guillou for topological discussions,
Patrick Massot for patiently explaining Seifert fibration to me,
and Pierre de la Harpe for both his comments on this article and his
talk during the fourteenth Tripode meeting in Lyon, which made me discover
the nice topic of Chabauty spaces.

\section{Types and norms}\label{sec:def}

We fix $n\in\bN$ and consider the Chabauty space $\chab=\chab(\bR^n)$.

Let $\dotprod$ denote the canonical scalar product on $\bR^n$ and
$|\phantom{x}|$ denote the corresponding Euclidean norm. It defines a distance not only on $\bR^n$,
but also on the Grassmannian $G(p;n)$ of all its $p$-dimensional sub-vector spaces
($p$-planes, for short). There are several classical ways to do this, but they do not differ for our purpose.

We denote under brackets $\langle\ \rangle$ the vector space generated by a subset
of $\bR^n$.

\subsection{Types}

Let $\Gamma$ be a point in $\chab$. It is isomorphic to $\bR^p\times\bZ^q$ for some integers
$p,q$. The pair $(p,q)$ is called the \emph{type} of $\Gamma$.
The \emph{rank} of $\Gamma$ is the dimension of the
vector space $\langle \Gamma\rangle$ it generates, that is $p+q$.
An element of rank $<n$ is said to be of lower rank. We denote by $\mathscr{R}_\ell$ the set of 
lower rank elements of $\chab$,
by $\mathscr{R}_m$ its complement and by $\chab^{(p,q)}$ the set of type $(p,q)$ elements.
The Lie group $\GL(n;\bR)$ acts naturally on $C$ and its orbits are exactly the $\chab^{(p,q)}$.

\subsection{Norms}

Let $\Gamma$ be a type $(p,q)$ point in $C$.
For all positive $r$,
let $\Gamma(r)$ be the subgroup of $\bR^n$ generated by $\Gamma\cap \overline{B}(0,r)$ where
$\overline{B}(0,r)$ is the closed ball of radius $r$ centered at the origin. Let
$\Gamma_0:=\cap_{r>0} \Gamma_r$ be the \emph{continuous part} of $\Gamma$. It is a $p$-plane 
of $\bR^n$. 

If $p>0$, then define
$$N_1(\Gamma)=\ldots=N_p(\Gamma)=0.$$
Let $r_1$ be the least
number $r$ such that $\Gamma(r)\neq \Gamma_0$ and $p_1$ be the rank of $\Gamma(r_1)$. Then define
$$N_{p+1}(\Gamma)=\ldots=N_{p_1}(\Gamma)=r_1.$$
Define similarly $(r_2,p_2),\ldots,(r_k,p_k)$ until $\Gamma(r_k)=\Gamma$.
Then one has $q=p_k-p$. At last, define 
$$N_{p+q+1}(\Gamma)=\ldots=N_n(\Gamma)=\infty.$$

The number $N_i(\Gamma)$ is called the $i$-th norm of $\Gamma$.
The norms are continuous
functions $N_i:C\to[0,\infty]$.

\subsection{Decomposition}

Any $\Gamma\in C$ has a \emph{canonical decomposition} $\Gamma=\Gamma_0+\Gamma_D$, where
$\Gamma_D:=\Gamma\cap\Gamma_0^\perp$ is a  discrete rank $q$ subgroup of $\bR^n$.

\begin{figure}[htp]\begin{center}
\input{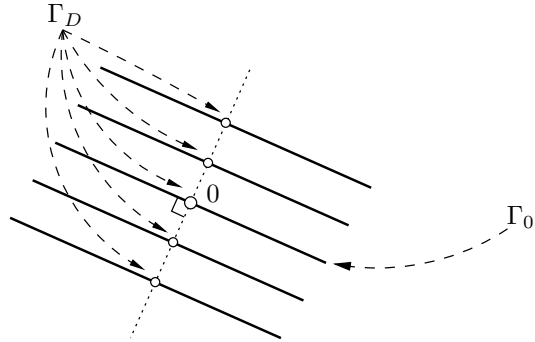}
\caption{Canonical decomposition of a type $(1,1)$ subgroup of $\bR^2$.}
\end{center}\end{figure}

\subsection{Combinatorial Structure}\label{sec:combin}

Let us consider the incidence scheme of the different $\chab^{(p,q)}$; these subsets
are natural ``strata'' of $\chab$.

If $n=1$, there are three types, namely those of $\bR$, $\bZ$ and the trivial group,
from now on denoted by $0$.
The closure of the second one contains the two other (both reduced to a point).
We sum this up into the diagram:
$$
\xymatrix@!=4pt{
    &\bZ \ar[ld]\ar[rd]&\\
\bR &                &  0
}
$$

If $n=2$, there are six types, organized according to the diagram:
$$
\xymatrix@!=4pt{
 & & \bZ^2\ar[ld]\ar[rd] & & \\
&\bR\times\bZ \ar[ld]\ar[rd]&& \bZ\ar[ld]\ar[rd]\\
\bR^2 &&  \bR && 0
}
$$

The diagram in the general case is:
$$
\xymatrix@!=4pt{
 & & & & \bZ^n\ar[dl]\ar[dr] & & & & \\
 & & & \bR\times\bZ^{n-1} \ar[dl] \ar[dr] & & \bZ^{n-1}\ar[dl]\ar[dr] & & &\\
 & & \cdots\ar[dl] & & \cdots & & \cdots\ar[dr] & &\\
&\bR^{n-1}\times\bZ\ar[dl]\ar[dr]&  & \ar[dl] & \cdots &\ar[dr] &  & \bZ\ar[dr]\ar[dl] & \\
\bR^n & &\bR^{n-1} &  &\cdots & & \bR & & 0 \\
}
$$
to be read as follows: the closure of the orbit of type $(r,s)$ intersects
the orbit of type $(p,q)$ if and only if there is a sequence of arrows
$\bR^r\times\bZ^s\to\cdots\to\bR^p\times\bZ^q$ (in which case we write
$(r,s)\geqslant(p,q)$).
A sequence of type $(r,s)$ elements can indeed converge to a point of a different type
in two (possibly simultaneous) ways: some of the non-zero, finite $N_i$ go to $0$ or to $\infty$.
In both cases $s$ decreases;
in the first one the rank is constant and $r$ increases while in the second one
$r$ is constant and the rank decreases. In other words, $(r,s)\geqslant(p,q)$ if and only if
$r\leqslant p$ and $r+s\geqslant p+q$: this is exactly what the diagram tells. 

Note that for example each (lower left)-(upper right) diagonal corresponds to the subgroups of a given rank.
In particular the largest of these diagonals corresponds to the set $R$ of higher rank subgroups.
Similarly, each (upper left)-(lower right) diagonal corresponds to the subgroups with continuous part of a given dimension,
in particular the largest of these diagonals correspond to the set of discrete subgroups.

\subsection{Duality}\label{sec:duality}

There is a well-known duality on the space of lattices of $\bR^n$. It extends word by word
to the larger space $\chab(\bR^n)$: the duality map
\begin{eqnarray*}
\ast : \chab(\bR^n) &\to& \chab(\bR^n) \\
         \Gamma &\mapsto&\Gamma^\ast=\ensemble{y\in\bR^n}{\forall x\in\Gamma\quad x\dotprod y \in\bZ}
\end{eqnarray*}
is an involutory homeomorphism. The dual of a type $(p,q)$ element is of type
$(n-(p+q),q)$. In particular, in the terminology to be introduced in the next section,
$\ast$ is a stratified isomorphism. On the above diagrams, duality induces a reflexion
with respect to a vertical axis and one only needs to understand half of the types to understand them all.

\begin{figure}[htp]\begin{center}
\input{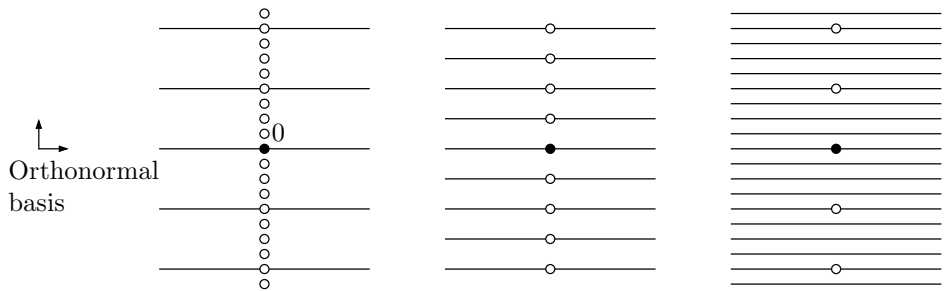}
\caption{In $\bR^2$, three type $(1,1)$ subgroups (dark lines) and their duals (white points).}
\end{center}\end{figure}

\section{Stratifications}\label{sec:strat}

There are many different types of stratifications; we shall use that of Mark Goresky and Robert 
MacPherson \cite{Goresky-MacPherson2}, but we also introduce more general definitions and that of 
Larry Siebenmann \cite{Siebenmann}.

\subsection{General definitions}

\begin{defi}\label{defi:stratif}
Let $X$ be a metrizable separable topological space. A \emph{stratification} of $X$ is a locally finite partition
$\mathscr{S}=(X^{(s)})_{s\in S}$ into locally closed subsets called \emph{strata} such that the
\emph{frontier condition} holds: for all $s,t$ in $S$, if $X^{(t)}\cap \overline{X}\null^{(s)}\neq\varnothing$
then $X^{(t)}\subset \overline{X}\null^{(s)}$. In other words, the closure of a stratum is a union of
strata. The couple $(X,\mathscr{S})$, often simply denoted by $X$, is called a \emph{stratified space}.
\end{defi}

In the works of Siebenmann and Goresky and MacPherson, the stratification are filtered by $\{0,1,\ldots,n\}$
rather than a more general set $S$. However, in the considered cases (CAT stratifications 
of finite dimension, see below) one can recover such a filtration, so that the above definition
is in fact consistent with \cite{Siebenmann} and \cite{Goresky-MacPherson2}.

The point in stratifying a space is to divide it into simple pieces, and the strata should not
be arbitrary for the stratification to be of interest. Most of the time, one asks that the strata
belong to a category of manifolds and think of stratified spaces as an extension of manifolds
that includes some singularities. The main motivation when Hassler Whitney introduced the first definition
of a stratified space was to study the topology of analytic varieties \cite{Whitney1,Whitney2,Whitney3}.
Ren\'e Thom used this concept to investigate the smooth maps between manifolds and their singularities
\cite{Thom}.
Denote respectively
by TOP, PL and DIFF the categories of topological, piecewise linear and smooth manifolds.

\begin{defi}
Let CAT be a category of manifolds (TOP, PL, or DIFF).
A stratification $(X^{(s)})_{s\in S}$ is a \emph{CAT stratification} if all strata are objects of CAT 
and $X^{(t)}\subset \overline{X}\null^{(s)}$ implies $\dim X^{(t)}<\dim X^{(s)}$.

The \emph{dimension} $d$ of a CAT stratification is the supremum of the dimensions of the strata (possibly
$\infty$). Its \emph{singular codimension} is the difference $d-d'$ where $d'$ is the second largest dimension
of the strata.
\end{defi}

The frontier condition is not very surprising since it is similar to a property 
of polyhedra (that is, topological realisation of simplicial complexes), where the closure
of a face is a union of faces. To show its particular relevance, let us introduce
a relation on the index set $S$ of a stratification. 
Given two strata $X^{(s)}$ and $X^{(t)}$, one writes $t\leqslant s$ if 
$X^{(t)}\cap\overline{X}\null^{(s)}\neq\varnothing$. It is easy to see that the frontier condition 
implies that $\leqslant$ is an ordering (the local closedness of strata is needed as well).

\subsection{Goresky-MacPherson stratifications}

The class of stratified spaces is stable under several natural operations. The fact
that the partitions given below are genuine stratifications is straightforward.

\begin{defi}
If $X$ and $Y$ are  stratified spaces with stratifications
$$\mathscr{S}=(X^{(s)})_{s\in S}\mbox{ and }\mathscr{T}=(Y^{(t)})_{t\in T}$$
the \emph{product}
$X\times Y$ is defined as the usual topological product endowed with the stratification
$$\mathscr{S}\times\mathscr{T}=\left(X^{(s)}\times Y^{(t)}\right)_{(s,t)\in S\times T}$$

Every open subset $U$ of $X$ inherits of the induced stratification $(U\cap X^{(s)})_{s\in S'}$
where $S'$ is the set of indices $s$ such that $X^{(s)}$ meets $U$.

The open cone $cX=X\times [0,1)/X\times\{0\}$ of a compact stratified space has a natural stratification, 
whose strata are the apex and the products $X^{(s)}\times(0,1)$. Such a $cX$ is called
a \emph{stratified cone}. The cone on the empty set is defined to be a point.

Recall that the \emph{join} of the topological spaces $X$ and $Y$ is the 
quotient of $X\times Y\times [0,1]$ by the relations
$(x,y,0)\sim (x',y,0)$ and $(x,y,1)\sim (x,y',1)$ (figure \ref{fig:join} shows examples).
We simply denote by $A\times\{1\}$ the image in this quotient of a set $A\times Y\times\{1\}$
when $A\subset X$.
When $X$ and $Y$ are stratified, their join $X\star Y$
can be endowed with a natural stratification.
Let $S\star T$ be the disjoint union of $S$, $T$ and
$S\times T$; the desired stratification $\mathscr{S}\star\mathscr{T}$ 
is indexed by $S\star T$, with strata
\begin{eqnarray*}
(X\star Y)^{(s)} &=& X^{(s)}\times\{1\} \\
(X\star Y)^{(t)} &=& Y^{(t)}\times\{0\} \\
(X\star Y)^{(s,t)} &=& X^{(s)}\times Y^{(t)}\times (0,1)
\end{eqnarray*}
Remark that the suspension of a space is simply its join with $S^0$, a pair of
distinct points.
\end{defi}

\begin{figure}[htp]\begin{center}
\includegraphics{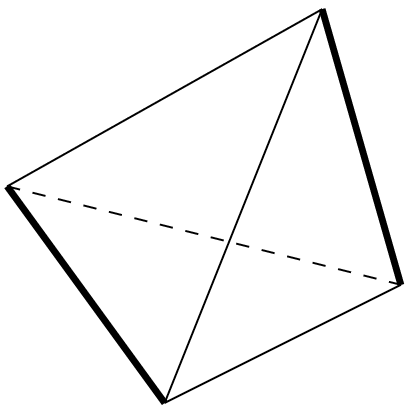}\hspace{2cm}\includegraphics{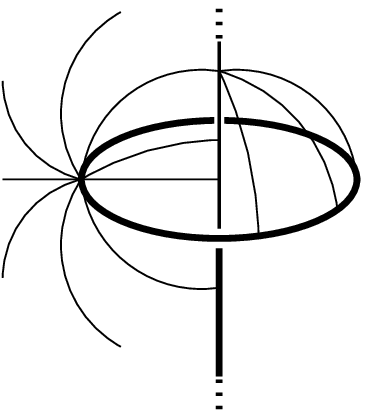}
\caption{The join of two segments is a $3$-simplex. The join of two circles is a $3$-sphere,
the contracted part being two Hopf fibers.}\label{fig:join}
\end{center}\end{figure}

Next we need to define isomorphisms.

\begin{defi}
If $X$ and $Y$ are stratified spaces,
a continuous map  $f:X\to Y$ is said to be \emph{stratified} if the inverse image of every stratum
of $Y$ is a stratum of $X$. It is a \emph{stratified isomorphism} (or simply, an isomorphism)
if it is stratified and a homeomorphism.

One defines in an obvious way \emph{PL} and \emph{DIFF stratified maps} and \emph{isomorphisms}.
\end{defi}

We are now able to introduce a local triviality condition on which the notions of
Siebenmann and Goresky-MacPherson stratified spaces rely.

\begin{defi}
A stratified space $X$ is \emph{locally cone-like} if for any point $x\in X$,
there is an open neighborhood $U$ of $x$ in its stratum $X^{(s)}$, a stratified cone $cL$ and an isomorphism
of $U\times cL$ onto an open neighborhood of $x$ in $X$ such that $U\times\{\mbox{apex}\}$ is mapped
identically onto $U$. The stratified space $L$ is called a \emph{link} of $x$ (it need not be
unique since there exist non homeomorphic spaces whose cones are homeomorphic).

A \emph{Siebenmann stratified space}
is a finite-dimensional, locally cone-like, TOP stratified space.

A \emph{Goresky-MacPherson stratified space} of dimension $n$ is defined recursively
as a $n$-dimensional Siebenmann stratified space, whose points admit links that
are lesser-dimensional Goresky-MacPherson stratified spaces.
One can define similarly PL and DIFF Siebenmann and Goresky-MacPherson stratified spaces.

A \emph{pseudo-manifold} is a Goresky-MacPherson stratified space where the union of maximal dimensional strata
is dense, and whose singular codimension is at least $2$.
\end{defi}

What we call a Siebenmann stratified space is a ``CS set'' in Siebenmann's terminology. Due to
the numerous definitions introduced by different authors, it seems better to use the authors' names
to distinguish between them. 
Note that in \cite{Goresky-MacPherson1}, contrary to \cite{Goresky-MacPherson2},
the stratified spaces considered are PL.

It seems to be an open question whether it exists a Siebenmann stratified space
that is not Goresky-MacPherson stratified.

Simple examples of Goresky-MacPherson stratified spaces are manifold with boundary (the strata
being the interior and the boundary) and polyhedron (stratified by their faces). A consequence
of Thom's ``first isotopy lemma'' is that analytic varieties can be endowed with a Goresky-MacPherson
stratification, see for example \cite{Goresky-MacPherson3} Section I.1.4.
In particular, complex analytic varieties are pseudo-manifolds.

\subsection{Some properties}

Compact Siebenmann stratified spaces have several nice properties. For example, their homeomorphism
groups are locally contractible. See \cite{Siebenmann} for more details. More important to us,
it gives a very natural way to describe $\chab=\chab(\bR^n)$ locally: the types will index the strata
$\chab^{(p,q)}$ and since their dimension is easy to compute, the description of neighbohoods of a point
reduces to a link.

Among Goresky-MacPherson stratified spaces, pseudo-manifolds are of utmost importance since they
have a so-called intersection homology satisfying some sort of Poincaré duality. It encodes in particular
the usual homology. Since it is not much more difficult to prove that $\chab(\bR^n)$ is
Goresky-MacPherson than to prove that it is Siebenmann, it seemed better to use this definition
even if we do not compute the intersection homology of $\chab(\bR^n)$. Note that
the codimensions of the strata need not be even, so that there is
no self-dual perversity for $\chab(\bR^n)$.

Let us turn to a remarkable fact: the product of two Siebenmann stratified space is Siebenmann
stratified. The only part that is not obvious in this statement is that a product
of two cones is again a cone.

\begin{lemm}\label{lemm:pdt_cones}
Let $X$ and $Y$ be stratified spaces. Then we have an isomorphism of stratified
spaces $cX\times cY\simeq c(X\star Y)$.
\end{lemm}

\begin{proof}
We write the elements of $cX$ in the form $(x,h)$ where $h\in [0,1[$ and $x\in X$, the
latter being meaningless when $h=0$.

The subset $\{((x,h),(y,\ell))\in cX\times cY;h+\ell<1\}$ is isomorphic to
$cX\times cY$, and is the cone over the subset
$\Delta=\{((x,h),(y,\ell))\in cX\times cY;h+\ell=\delta\}$ for any $\delta<1$.

To see that $\Delta$ is isomorphic to $X\star Y$, simply consider the map
\begin{eqnarray*}
X\times Y\times [0,\delta] &\to& \Delta \\
(x,y,m)&\mapsto& ((x,m),(y,\delta-m))
\end{eqnarray*}
(see figure \ref{fig:cones}).
\end{proof}

\begin{figure}[htp]\begin{center}
\input{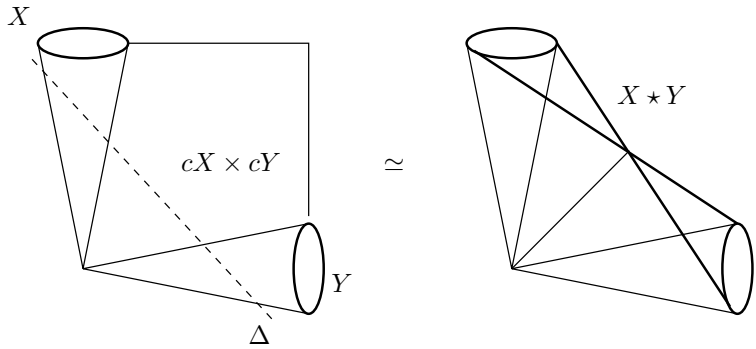}
\caption{The product of cones is a cone.}\label{fig:cones}
\end{center}\end{figure}

Since a join is locally a product (possibly involving a cone), the join of two
Siebenmann stratified spaces is again Siebenmann stratified. Then a stratightforward
induction leads to~: the product, the join and the cone of Goresky-MacPherson stratified
spaces are Goresky-MacPherson stratified. 

\subsection{Stratified bundles}

Some of the links in $\chab(\bR^n)$ shall be described as some sort of fiber bundles, where
the fiber depends upon the strata. Note that we define such bundles only in
the category of Goresky-MacPherson stratified spaces.

\begin{defi}\label{defi:bundle}
A \emph{stratified bundle} is defined inductively as a surjective continuous map 
$\pi : E\to B$ where:
\begin{itemize}
\item $E$ is a metrizable separable topological space called the \emph{total space},
\item $B$ is a Goresky-MacPherson stratified space (with stratification
      $(B^{(s)})_{s\in S}$) called the \emph{base},
\item there are topological manifolds $F_s$ called the \emph{fibers} such that 
      all $x\in B$ (in the $s$ strata, with link $L$ say) has a conical neighborhood
      $V\simeq \bR^k\times cL$ such that $\pi^{-1}(V)\simeq \bR^k\times F_s\times cL'$
      where $L'$ is a compact Goresky-MacPherson stratified space and $\pi$ writes in the form
      $$\pi(b,f,(t,l'))=(b,(t,\pi'(l'))) \quad \forall b\in \bR^k, \forall f\in F_s, \forall (t,l')\in cL'$$
      where $\pi':L'\to L$ is a stratified bundle. 
\end{itemize}
One defines similarly PL or DIFF stratified bundles in the category of PL or DIFF Goresky-MacPherson 
stratified spaces.
\end{defi}

We could have assumed $E$ to be a Goresky-MacPherson stratified space and $\pi$ to be a stratified
map, but this is not necessary.

\begin{lemm}
If $\pi :E\to B$ is a stratified bundle, the partition $E^{(s)}=\pi^{-1}(B^{(s)})$ is a Goresky-MacPherson
stratification. 
\end{lemm}

\begin{proof}
First, let us show that the frontier condition holds. Let $s$ and $t$
be indices such that $E^{(t)}\cap \overline{E}\null^{(s)}\neq \varnothing$ and
$z\in E^{(t)}$. Then $B^{(t)}\subset \overline{B}\null^{(s)}$ and there is
a sequence $x_n\in B^{(s)}$ that converges to $\pi(z)$.Thanks to the local
form of $\pi$, we can lift $x_n$ to a sequence $z_n\in E^{(s)}$ that converges
to $z$. 

Local closedness of strata and local finiteness of the partition are
direct consequences of the definition, as well as the strata being topological manifolds
and satisfying the dimension condition. The partition $(E^{(s)})_{s\in S}$ is
therefore a TOP stratification.

The local form of $\pi$ implies readily that it is also a Siebenmann stratification,
and an induction on the dimension of $B$ shows at last that it is Goresky-MacPherson.
\end{proof}

This definition of a stratified bundle is a generalization of fiber bundles over manifolds,
since the restriction of $\pi$ to $E^{(s)}\to B^{(s)}$ is a fiber bundle in the usual sense
for all $s\in S$. It is however quite restrictive, in particular the family of fibers cannot
be arbitrary: if $s\geqslant t$, then
$F_s$ must be homeomorphic to $F_t\times F'_s$ where $F'_s$ is the fiber over
$L^{(s)}$ for the bundle $\pi'$. In the stratified bundles that appear in the local study of
$\chab(\bR^n)$, the fibers are tori whose dimension depends upon the strata.

\section{Local study of $\chab(\bR^n)$}\label{sec:local}

\subsection{Stratification of $\chab(\bR^n)$}

Let us start with the simplest part of Theorem \ref{theo:stratif}.

\begin{prop}\label{prop:stratif2}
The partition $(\chab^{(p,q)})_{(p,q)\in S}$ is a DIFF stratification of $\chab$.
\end{prop}

\begin{proof}
First, $\chab$ is known to be metrizable and compact (see for example \cite{delaHarpe}).

To see that the strata are locally closed, it is sufficient to have a look at
a neighborhood $U$ of a point $\Gamma\in \chab^{(p,q)}$ that is sufficiently small
to ensure that 
$$N_{p+1},\ldots,N_{p+q}\in(0,\infty)$$
for all elements of $U$. There the strata is defined by the equations $N_1=\cdots=N_p=0$ and
$N_{p+q+1}=\cdots=N_n=\infty$, thus $U\cap \chab^{(p,q)}$ is closed in $U$.

The frontier condition, anyway simple to get from a direct study, comes for free from 
the description of strata as orbits of the action
of $\GL(n;\bR)$: if $\chab^{(p,q)}$ intersects $\overline{\chab}\null^{(r,s)}$, then there is a sequence
$\Gamma_n\in \chab^{(r,s)}$ that converges to some $\Gamma_\infty\in \chab^{(p,q)}$. For
all $\Gamma\in \chab^{(p,q)}$ there is a $g\in\GL(n;\bR)$ such that $\Gamma=g(\Gamma_\infty)$,
and the sequence $g(\Gamma_n)$ converges to $\Gamma$, hence $\chab^{(p,q)}\subset \overline{\chab}\null^{(r,s)}$.

We also get the manifold structures on strata from this action: for all $(p,q)$, the stabilizer
of the element $\bR^p\times\bZ^q=e_1\bR+\cdots+e_p\bR+e_{p+1}\bZ+\cdots+e_{p+q}\bZ$ (where $(e_i)$ is the canonical basis of
$\bR^n$) of $\chab^{(p,q)}$ is a closed subgroup $H_{(p,q)}$ of $\GL(n;\bR)$, thus a Lie subgroup.
We can endow $\chab^{(p,q)}$ with the manifold structure of $\GL(n;\bR)/H_{(p,q)}$.

Last, the dimension of $\chab^{(p,q)}$ is easily computed: the continuous part
is an element of $G(p;n)$, which has dimension $p(n-p)$, and the discrete part
is defined by the choice of $q$ vectors in a $(n-p)$ plane. We get that
$$\dim \chab^{(p,q)} = (p+q)(n-p)$$
in particular $\dim \chab^{(r,s)} > \dim \chab^{(p,q)}$ as soon as $(r,s)>(p,q)$.
\end{proof}

Note that in the sequel it will be simpler to prove only the TOP stratification of links,
so we mainly think of $\chab(\bR^n)$ as a TOP stratified space.

\subsection{Decomposition at a given scale}

Let us introduce a number of definitions to be used in the next subsection. They aim to
give a parametrization of neighborhoods in $\chab(\bR^n)$, by decomposing subgroups at three scales.
We could define more general definitions, involving more different scales but the following is
sufficient for our purpose.

\subsubsection{$\delta$-decomposability}

A \emph{scale} is a number $\delta\in(0,1)$, usually small. An element $\Gamma\in \chab(\bR^n)$
is said to be \emph{decomposable at scale $\delta$} if for all $i$, $N_i(\Gamma)\notin\{\delta,\delta^{-1}\}$.
We then say that $\Gamma$ has \emph{$\delta$-type} $(p,q)$ if
$$\left\{\begin{array}{rcl}
N_1(\Gamma),\ldots,N_p(\Gamma) &<& \delta \\
N_{p+1}(\Gamma),\ldots,N_{p+q}(\Gamma) &\in& (\delta,\delta^{-1}) \\
N_{p+q+1}(\Gamma),\ldots, N_n(\Gamma) &>& \delta^{-1}
\end{array}\right.$$
Note that the $\delta$-type of $\Gamma$ is always at most its type (with respect to the order of Section
\ref{sec:combin}, that is the order given by the frontier condition).

\subsubsection{Local trivialisation}

The motivation for this paragraph is the following. We shall associate 
to a $\delta$-decomposable element a triple of vector spaces, generated 
by three parts of $\Gamma$ (one at small scale, one at medium scale and one at large scale).
To compare close $\delta$-decomposable elements, we need to fix an identification between close
subspaces of $\bR^n$.

A \emph{linear decomposition} of type $(p,q)$ of $\bR^n$ is a triple $(V_1,V_2,V_3)$ where
$V_1$ is a $p$-plane, $V_2$ is a $q$-plane, $V_3$ is a $(n-(p+q))$-plane and
$$\bR^n=V_1\stackrel{\perp}{\oplus} V_2\stackrel{\perp}{\oplus} V_3$$
A linear decompositions of type $(p,q)$ can be naturally identified with the $(p,p+q)$ 
flag $(V_1,V_1+V_2)$. We therefore denote by $G(p,p+q;n)$ the set of all type $(p,q)$
linear decomposition. It is a manifold, and inherits a metric from the Euclidean structure
of $\bR^n$.

Given a type $(p,q)$ linear decomposition $(V_1^0,V_2^0,V_3^0)$, there is a small ball $V$ in
$G(p,p+q;n)$ centered at $(V_1^0,V_2^0,V_3^0)$ and a small ball $U$ around the identity in
a submanifold of $\SO(n)$ such that for all $(V_1,V_2,V_3)$ in $V$, there is a unique $\tau\in U$ (called
the trivialisation of $(V_1,V_2,V_3)$) such that
$\tau(V_1,V_2,V_3)=(V_1^0,V_2^0,V_3^0)$. Moreover the mapping $(V_1,V_2,V_3)\mapsto \tau$
can be chosen a diffeomorphism. From now on, we assume that for all $(V_1^0,V_2^0,V_3^0)$
we have chosen such a mapping (called a \emph{local trivialisation}).

Let $\Gamma$ be a $\delta$-decomposable element of $\delta$-type $(p,q)$. The
\emph{linear decomposition} (at scale $\delta$) of $\Gamma$ is defined as follows.
First, $V_1=\langle\Gamma(\delta)\rangle$ is the $p$-plane generated by the element
of $\gamma$ of norm less than $\delta$. We denote by $P'_2$ the orthogonal
projection on $V_1^\perp$. Then $V_2=\langle P'_2\Gamma(\delta^{-1})\rangle$ is
a $q$-plane orthogonal to $V_1$. At last, $V_3$ is defined as $(V_1+V_2)^\perp$,
and by construction $(V_1,V_2,V_3)$ is a linear decomposition.

\subsubsection{Parametrization of a neighborhood}\label{sec:parameter}

Let us define a parametrization of a neighborhood of a type $(p,q)$ element $\Gamma^0$
in $\chab(\bR^n)$. Let $V_1^0=\Gamma^0_0$ be its continuous part and $V_2^0=\langle\Gamma^0_D\rangle$
be the $q$-plane generated by its discrete part (which is orthogonal to $V_1^0$). We define
$V_3^0=(V_1^0+V_2^0)^\perp$ and we assume that a basis $(e_{p+1},\ldots,e_{p+q})$
of $\Gamma_D$ has been fixed. In what follows the dependence on this basis is not crucial.
For convenience, we also assume that we have fixed linear isomorphisms $V_0^1\simeq \bR^p$,
$\Gamma^0_D\simeq\bZ^q$ (identifying $(e_{p+1},\ldots,e_{p+q})$ with the canonical basis)
and $V_0^1\simeq \bR^{n-(p+q)}$. We may use this identifications without notice.

Choose a small scale $\delta$. The required smallness will be precised at several steps
below. First we assume that $\delta< N_{p+1}(\Gamma^0)$ and $\delta^{-1}>N_{p+q}(\Gamma^0)$,
so that $\Gamma^0$ has $\delta$-type $(p,q)$ and $(V_1^0,V_2^0,V_3^0)$ is its linear
decomposition at scale $\delta$.

Then define $U$ as the set of all $\Gamma\in\chab(\bR^n)$
such that:
\begin{itemize}
\item $\Gamma$ is $\delta$-decomposable,
\item its linear decomposition $(V_1,V_2,V_3)$ is $\delta$-close to $(V^0_1,V^0_2,V^0_3)$,
\item denoting by $\tau$ the corresponding trivialisation and by $P_2$ the orthogonal
      projection onto $V_2$, $\tau P_2(\Gamma(\delta^{-1}))\subset V_2^0$ is generated
      by vectors $v_{p+1},\ldots,v_{p+q}$ such that $|e_i-v_i|<\delta$ for all $i$,
\end{itemize}
It is an open neighborhood of $\Gamma^0$.
From a $\Gamma\in U$ we construct its \emph{local
decomposition} $(\Gamma_1,\Gamma_2,\Gamma_3,\Phi_2,\Phi_3)$ as follows.
First, $\Gamma_1=\tau\Gamma(\delta)$ is a closed subgroup of $V_1^0\simeq\bR^p$ of maximal rank
in $\chab(\bR^p)$. Second, $\Gamma_2=\tau P_2(\Gamma(\delta^{-1}))$
is a discrete subgroup of $V_2^0$ close to $\Gamma_D\simeq\bZ^q$. The \emph{distinguished
basis} of $\Gamma_2$ is the basis $v_{p+1},\ldots,v_{p+q}$ that satisfies
$|e_i-v_i|<\delta$ for all $i$ (we assume $\delta$ is small enough to ensure that
this basis is uniquely defined). This distinguished basis defines an identification between
$\Gamma_2$ and $\bZ^q$.
 Third, denoting by $P_3$ the orthogonal projection onto $V_3$,
$\Gamma_3=\tau P_3(\Gamma)$ is a discrete subgroup of $V^0_3\simeq\bR^{n-(p+q)}$.

Now $\Phi_2:\bZ^q\to V_1^0/\Gamma_1$ is the unique homomorphism such that
$\Gamma(\delta^{-1})$ is generated by the sets $\tau^{-1}(v_i+\Phi_2(e_i))$ for $p<i\leqslant p+q$.
Note that here we consider $\Phi_2(e_i)$ as a $\Gamma_1$ coset in $V_1^0$. Figure \ref{fig:coord}
illustrates this map.
Last, $\Phi_3:\Gamma_3\to (V_1^0+V_2^0)/(\Gamma_1+\Gamma_2)$ (where the range
can be identified with $\bR^p/\Gamma_1\times\bR^q/\bZ^q)$) is the unique homomorphism
such that $\Gamma$ is the union of the $\tau^{-1}(v+\Phi_3(v))$ for $v\in\Gamma_3$.
It can be useful to further decompose $\Phi_3$ into $\Phi_3':\Gamma_3\to V_1^0/\Gamma_1$
and $\Phi_3'':\Gamma_3\to V_2^0/\Gamma_2\simeq\bR^q/\bZ^q$.

\begin{figure}[htp]\begin{center}
\input{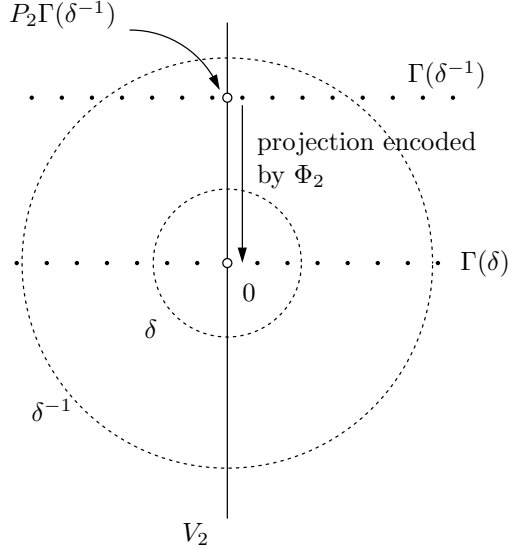}
\caption{The map $\Phi_2$ enables the recovering of $\Gamma(\delta^{-1})$ from $\Gamma_1$,
$\Gamma_2$ and the local trivialisation $\tau$ (in this example, $\Gamma$ has $\delta$-type $(1,1)$).}\label{fig:coord}
\end{center}\end{figure}

From the linear and local decompositions of $\Gamma$, it is easy to reconstruct
$\Gamma$. The map 
$$\Gamma\mapsto\big((V_1,V_2,V_3),(\Gamma_1,\Gamma_2,\Gamma_3,\Phi_2,\Phi_3)\big)$$
and its inverse are moreover continuous.

It will be simpler in the sequel to use the part of
$U$ defined by 
$$N_p(\Gamma_1)+N_{p+q+1}(\Gamma_3)^{-1}<\delta$$
(note that
$N_{p+q+1}(\Gamma_3)$ is the first norm of $\Gamma_3$ viewed as an element
of $\chab(\bR^{n-(p+q)})$). This neighborhood is denoted by $U_\delta^n(\Gamma^0)$.

\subsection{Neighborhoods in $\chab(\bR^n)$}

Now are now ready to prove Theorem \ref{theo:stratif}.

\begin{lemm}\label{lemm:Ln0}
The trivial subgroup $0\in \chab(\bR^n)$ has a neighborhood of the form
$cL^n(0)$ where the link $L^n(0)$ is the set of subgroups of unit (first) norm,
stratified by its intersection with the strata of $\chab(\bR^n)$.
\end{lemm}

\begin{proof}
The neighborhood $U_1^n(0)$ defined above is the set of elements
of norm greater than $1$.
The map
\begin{eqnarray*}
[0,1) \times L^n(0) &\to& U_1^n(0)\\
(t,\Gamma)&\mapsto&  t^{-1}\Gamma
\end{eqnarray*}
(with the convention $\infty\Gamma=\Gamma_0$, here $\Gamma_0=0$)
is continuous and induces a homeomorphism $cL^n(0)\to U_1^n(0)$.

Since $U_1^n(0)\smallsetminus \{0\}\simeq (0,1)\times L^n(0)$ is open, it inherits a stratification
from that of $\chab(\bR^n)$. It follows that the intersections of $L^n(0)$
with the strata $\chab^{(p,q)}$ does define a stratification. The above homeomorphism
becomes a stratified isomorphism when $L^n(0)$ is given this stratification.
\end{proof}
Note that we could do the same with any fixed value for the first norm instead of $1$.

The local study of the total group follows immediately from that of $0$.
\begin{lemm}\label{lemm:LnRn}
The total group $\bR^n\in \chab(\bR^n)$ has a neighborhood 
$cL^n(\bR^n)$ where $L^n(\bR^n)\simeq L^n(0)$ is the set of subgroups
of $n$-th norm $1$.
\end{lemm}

\begin{proof}
The duality map $*$ is a stratified isomorphism and maps $0$ to $\bR^n$.
It must therefore map $U_1^n(0)$ onto a neighborhood of $\bR^n$.
We can also reproduce the proof of Lemma \ref{lemm:Ln0}: the neighborhood
$U_1^n(\bR^n)$ is the set of elements of $n$-th norm at most $1$
and the map
\begin{eqnarray*}
[0,1) \times L^n(\bR^n) &\to& U_1^n(\bR^n)\\
(t,\Gamma)&\mapsto&  t\Gamma
\end{eqnarray*}
(with the convention $0\Gamma=\langle \Gamma\rangle$, here $\langle \Gamma\rangle=\bR^n$) 
induces an isomorphism $cL^n(\bR^n)\to U_1^n(\bR^n)$.
\end{proof}

\begin{lemm}
Any type $(p,q)$ element $\Gamma^0\in \chab(\bR^n)$ 
has a neighborhood
of the form $\bR^{(n-p)(p+q)}\times cL$
where the link $L=L^n(\Gamma^0)$
is defined in $U_\delta^n(\Gamma^0)$ (where $\delta$ is any small
enough scale) by the equations
\begin{eqnarray*}
(V_1,V_2,V_3)&=&(V_1^0,V_2^0,V_3^0) \\
\Gamma_2&=&\Gamma^0_D\\
N_p(\Gamma_1)+N_{p+q+1}(\Gamma_3)^{-1}&=&\delta/2
\end{eqnarray*}
where we use the notations of Section \ref{sec:parameter}.
\end{lemm}
Implicitely, $L$ is stratified by its intersection with the strata of $\chab(\bR^n)$.

\begin{proof}
By homogeneity of strata we can restrict to $\Gamma^0=\bR^p\times \bZ^q$. 
Consider its neighborhood $U_\delta^n(\bR^p\times\bZ^q)$:
an element $\Gamma$ there has a linear decomposition at scale $\delta$
$(V_1,V_2,V_3)$ and a local decomposition 
$(\Gamma_1,\Gamma_2,\Gamma_3,\Phi_2,\Phi_3)$.
Its projection to $\chab^{(p,q)}$ is defined as
$V_1+\tau^{-1}(\Gamma_2)\subset V_1+V_2$. It can be arbitrary
in a neighborhood of $\bR^p\times \bZ^q$ in $\chab^{(p,q)}$.
As a consequence, $U_\delta^n(\bR^p\times\bZ^q)$ is
isomorphic the product of two sets, the set
of possible choice of $(V_1,V_2,V_3,\Gamma_2)$,
which is a $(n-p)(p+q)$-dimensional ball, and the set $M$
of possible choices of $(\Gamma_1,\Gamma_3,\Phi_2,\Phi_3)$.
This last set is of course stratified by the type $(r,s)$ of the corresponding
point $\Gamma$. This type depends only upon $(\Gamma_1,\Gamma_3)$.

Apart from the choice of $\Phi_2$ and $\Phi_3$, $M$ looks like
the product of two cones $cL^p(\bR^p)$ and $cL^{n-(p+q)}(0)$, and we proceed
as in the proof of Lemma \ref{lemm:pdt_cones}. The map
\begin{eqnarray*}
[0,2) \times L^n(\bR^p\times \bZ^q) &\to& M\\
(t,\Gamma_1,\Gamma_3,\Phi_2,\Phi_3)&\mapsto&  (t\Gamma_1,t^{-1}\Gamma_3,t\Phi_2,\Phi_3^t)
\end{eqnarray*}
(where $\Phi_3^t(t^{-1}\gamma):=(t\Phi_3'(\gamma),\Phi_3''(\gamma))$ for all $\gamma\in\Gamma_3$)
induces the required isomorphism $cL^n(\bR^p\times \bZ^q) \to M$.
\end{proof}
Note that Lemmas \ref{lemm:Ln0}, \ref{lemm:LnRn} are included in this result.
We now can tell that $\chab(\bR^n)$ is Siebenmann stratified, but we can get more.

\begin{lemm}\label{lemm:core}
Let $\Gamma_0\in \chab(\bR^n)$ and consider an element $\Gamma\in\chab(\bR^n)$
that lies on the link $L^n(\Gamma^0)$. For small enough $\delta$, the neighborhood
$U_\delta^n(\Gamma)\cap L^n(\Gamma^0)$ of $\Gamma$ in $L^n(\Gamma^0)$ is of the form
$\bR^k\times cL^n(\Gamma)$ (where $k$ depends on the types of $\Gamma^0$ and $\Gamma$).
\end{lemm}

\begin{proof}
This follows directly from previous lemma. Let $(p,q)$ and $(r,s)$ be the types of
$\Gamma^0$ and $\Gamma$. Up to a change of scale
$cL^n(\Gamma)\subset L^n(\Gamma^0)$ and $L^n(\Gamma^0)$ intersects
the $(r,s)$ strata of $U_\delta^n(\Gamma)$ along a submanifold of $\bR^{(n-r)(r+s)}$. 
\end{proof}

The following last lemma settles the proof of Theorem \ref{theo:stratif} and shows
how the general links are related to the $L^k(0)$.
\begin{lemm}\label{lemm:bundle}
For all $(p,q)$, the link $L=L^n(\bR^p\times\bZ^q)$ is a Goresky-MacPherson
stratified space.

Moreover, if $(p,q)$ is different from $(0,0)$ and $(n,0)$, then the map
\begin{eqnarray*}
\pi=\pi(n,p,q): L=L^n(\bR^p\times\bZ^q) &\to& L^p(\bR^p)\star L^{n-(p+q)}(0)\\
(\Gamma_1,\Gamma_3,\Phi_2,\Phi_3) &\mapsto& 
  \left(N_p(\Gamma_1)^{-1}\cdot\Gamma_1;N_{p+q+1}(\Gamma_3)\cdot\Gamma_3;\frac2\delta N_p(\Gamma_1) \right)
\end{eqnarray*}
is a stratified bundle
and for all $(r,s)>(p,q)$
the fiber over the stratum $L^{(r,s)}$ is a torus of dimension
$$q(p-r)+(r+s-p-q)(p+q-r)$$
\end{lemm}

\begin{proof}
The proof of the first part is by decreasing induction on 
$(p,q)$, with respect to the usual ordering obtained from the condition of frontier.

If $(p,q)=(0,n)$, the link
$L^n(\bZ^n)$ is empty. If $(p,q)<(0,n)$, Lemma \ref{lemm:core} shows
that the link $L=L^n(\bR^p\times\bZ^q)$ is a cone-like TOP stratified space with 
links of the form $L^n(\bR^r\times\bZ^s)$ with $(r,s)>(p,q)$, which are Goresky-MacPherson
stratified by induction hypothesis.

For the second part, first remark that the map $\pi$ restricts to a 
(classical) fiber bundle on each strata, and the
fibers corresponds to the choice of $\Phi_2$ and $\Phi_3$ when given
$\Gamma_1$ and $\Gamma_3$. Each of these maps is defined by the image in a torus
(of respective dimension $p-r$ and $p+q-r$) of a basis of a lattice
(of respective rank $q$ and $(r+s-p-q)$); this gives the claimed topology for the fibers.

Next we proceed by a similar induction than above. 
If $(p,q)\leqslant(0,n)$ and $(\Gamma_1,\Gamma_3,\lambda)$ is a point in  $L^p(\bR^p)\star L^{n-(p+q)}(0)$,
then $(r,s)>(p,q)$ where $r$ is the dimension of the continuous part of $\Gamma_1$ and
$r+s-p-q$ is the rank of $\Gamma_3$. If $U$ is a small enough neighborhood
of $(\Gamma_1,\Gamma_3,\lambda)$ and $V=\pi^{-1}(U)$, then the restriction
of $\pi$ to $V\to U$ writes in the form required by Definition \ref{defi:bundle}
with $\pi'=\pi(n,r,s)$, thus is a stratified bundle by induction hypothesis.
\end{proof}

Note that the singular codimension of $\chab(\bR^n)$ is $n$, thus it is a pseudo-manifold if $n\geqslant 2$:
the proof of Theorem \ref{theo:stratif} is over.

The following will be central in the proof of Theorem \ref{theo:main} (recall 
that $\mathscr{R}_m$ is the subset of rank $n$ element of $\chab(\bR^n)$ and $\mathscr{R}_\ell$
is its complement).

\begin{coro}\label{coro:system}
Any type $(p,q)$ element $\Gamma\in \chab(\bR^n)$ has a neighborhood 
system $(U_\varepsilon)_\varepsilon$ such that $U_\varepsilon$ is contractible
and $U_\varepsilon\smallsetminus \mathscr{R}_\ell$ is pathwise connected.
\end{coro}

\begin{proof}
First, the links $L^n(0)$ are pathwise connected. Indeed $L^n(0)\cap\mathscr{R}_m$
is a dense strata, so that any point in $L^n(0)$ can be connected to a point
in $L^n(0)\cap\mathscr{R}_m$, which is pathwise connected (homeomorphic to $\GL(n;\bR)/\GL(n;\bZ)$).
 
Taking a neighborhood of $\Gamma$ isomorphic to $U=\bR^{(n-p)(p+q)}\times cL$ where $L=L^n(\bR^p\times\bZ^q)$
and for $U_\varepsilon$ ($\varepsilon\in (0,1)$) the product of the radius $\varepsilon$
ball in $\bR^{(n-p)(p+q)}$ by the part $\ensemble{(t,l)\in cL}{t<\varepsilon}$
of the cone $cL$, we get a neighborhood system such that 
$U_\varepsilon$ is contractible and $U_\varepsilon\smallsetminus \mathscr{R}_\ell$ is a deformation
retract of the total space of a stratified bundle with tori as fibers and
base $$L^p(0)\star \big(L^{n-(p+q)}(0)\big)^{(0,n-(p+q))}$$
The tori are pathwise connected as well as $L^p(0)$. Moreover
$(L^{n-(p+q)}(0)\big)^{(0,n-(p+q))}$ is the set of unit norm lattices
in $\bR^{n-(p+q)}$ and is therefore pathwise connected. The 
pathwise connectedness
of $U_\varepsilon\smallsetminus \mathscr{R}_\ell$ follows.
\end{proof}

\subsection{Complete description of a few links}

Let us consider some explicit examples. We use Lemmas \ref{lemm:Ln0}, \ref{lemm:LnRn}
and \ref{lemm:bundle}.
First, as already noticed, 
$$L^1(0)\simeq\{\bZ\}$$
is reduced to a point  and the case of the open strata is trivial: 
$$L^n(\bZ^n)=\varnothing$$
for all $n$.

\subsubsection{Description of links when $n=2$}

We already said that 
$$L^2(0)\simeq L^2(\bR^2)\simeq S^3$$
where $S^3$ is stratified with strata a trefoil knot and its complement. The proof is contained
in that of the Hubbard-Pourezza theorem, see the last section.

The duality maps type $(1,1)$ elements to type $(0,1)$ ones, so that
we have left to consider only $L^2(\bZ)$ and $L^2(\bR)$.
The link $L^2(\bZ)$ is isomorphic to the set of couples $(\Gamma_3,\Phi_3)$
where $\Gamma_3\in L^1(0)\simeq \{\bZ\}$ and $\Phi_3$ is a homomorphism
$\Gamma_3\to \bR/\bZ$. As a consequence, 
$$L^2(\bZ)\simeq S^1$$
where $S^1$ is stratified with one strata.

The link $L^2(\bR)$ is isomorphic to the set of triples
$(\Gamma_1,\Gamma_3,\Phi_3)$ where $\Gamma_1\in U^1_\varepsilon(\bR)$
is defined by its norm $\alpha$, $\Gamma_3\in U^1_\varepsilon(0)$
is defined by the inverse $\beta$ of its norm,
$\Phi_3$ is a homomorphism from $\Gamma_3\simeq \beta^{-1}\bZ$
to $\bR/\Gamma_1\simeq \bR/\alpha\bZ$, and moreover
$\alpha+\beta$ is constrained to be equal to a constant $\delta/2$.
As a consequence, 
$$L^2(\bR)\simeq S^2$$
where $S^2$ is stratified
with three strata, two of them being reduced to a point (see figure \ref{fig:link}).

\begin{figure}[htp]\begin{center}
\input{link_1_0_2.pstex_t}
\caption{The link $L^2(\bR)$.}\label{fig:link}
\end{center}\end{figure}

\subsubsection{Description of links $L^n(\bZ^{n-1})$}

The link $L^3(\bZ^2)$ is isomorphic to the set of
couples $(\Gamma_3,\Phi_3)$ such that $\Gamma_3\in L^1(0)\simeq\{1\}$
and $\Phi_3$ is a homomorphism from $\bZ$ to $\bR^2/\bZ^2$.
Therefore 
$$L^3(\bZ^2)\simeq T^2$$
where $T^2$ is the $2$-torus stratified with a single stratum. This case is important,
since it shows very simply that $\chab(\bR^3)$ is not a manifold:
$\bZ^2$ has a neighborhood homeomorphic to
$\bR^6\times cT^2$.
The same argument shows 
$$L^n(\bZ^{n-1})\simeq T^{n-1}$$
thus the same conclusion holds for all $n>2$. We see that
$\chab(\bR^2)$ is a manifold only because of a ``happy accident'':
the cone over $T^1$ is a $2$-ball.

\subsubsection{Description of some links in $\chab(\bR^3)$}

Let us give a few more examples without details.
We have
$$L^3(\bR\times\bZ)\simeq T^1\times T^1\times T^1 \times [0,1] /\sim$$
where the quotient is by the relations
$(x,y,z,0)\sim (x',y',z,0)$ and $(x,y,z,1)\sim (x,y',z',1)$.

There is a stratified bundle
$$L^3(\bZ) \to S^3$$
where $S^3$ is stratified by a trefoil knot and its complement, the fibers
of this bundle being $T^2$ (over generic points) and $T^1$ (over singular points).

There is a stratified bundle
$$L^3(\bR)\to \bar{c}S^3$$
where $S^3$ is again stratified by a trefoil knot and its complement and  
$\bar{c} S^3=\{\bullet\}\star S^3$ is the closed cone over $S^3$, the fibers
of this bundle being $T^2$  (over generic points), $T^1$ (over $K\times(0,1)$)
and a point (over the apex of the cone and $S^3\times\{1\}$).

\section{Localization and simple connectedness}\label{sec:global}

In this section we prove our main result. We start with the localization theorem \ref{theo:localization}
and then prove that $\chab(\bR^n)$ is simply connected.

\subsection{Localization}

Let $X$ be a Hausdorff topological space, 
$Y$ be a closed subset of $X$ and $m$ be any non-negative
integer.
Assume that each point $y\in Y$ has an neighborhood system $(U_\varepsilon)$
in $X$ such that the topological pair 
$(U_\varepsilon,U_\varepsilon\smallsetminus Y)$ is $m$-connected.
Let us prove by induction on $k\leqslant m$ that the pair $(X,X\smallsetminus Y)$ is $k$-connected.
In fact, we shall prove a stronger property to run the induction.

For all $k\leqslant m$, we denote by $I^k$ the cube $[0,1]^k$ and by $\partial I^k$ its boundary,
while $0$ denotes the point $(0,0,\ldots,0)$.

Fix some point $x_0\in X\smallsetminus Y$ and let us prove by induction on $k$ that any map 
$\alpha: (I^k,\partial I^k,0)\to (X,X\smallsetminus Y,x_0)$ is homotopic
(with fixed boundary) to a map $\alpha_1: I^k\to X\smallsetminus Y$ through
a arbitrarily small homotopy $(\alpha_t)$. More precisely,
we shall prove that for any compact subsets $K_1,\ldots,K_\ell$ of
$I^k$ and any open subsets $W_1,\ldots,W_\ell$ of $X$ such that
$\alpha(K_j)\subset W_j$, we can ask that for all $t$,
$\alpha_t(K_j)\subset W_j$.

For any point $x$ in the interior of $I^k$, an \emph{open box} around $x$ is a neighborhood of $x$
that writes $I_1\times I_2\times\cdots\times I_k$ where $I_i$ are open intervals
of $[0,1]$ that contain neither $0$ nor $1$.
The \emph{lower corner} of an open box $]a_1,b_1[\times\cdots\times]a_k,b_k[$
is the point $(a_1,a_2,\ldots,a_k)$.

We start with the case $k=0$. We have to prove that any point
in $X$ is connected by an arbitrarily small path to a point in $X\smallsetminus Y$.
This follows directly from the hypothesis that $U_\varepsilon$ is a neighborhood system
and $(U_\varepsilon,U_\varepsilon\smallsetminus Y)$ is $0$-connected.

Assume now that we proved the desired result for maps 
$(I^{k-1},\partial I^{k-1},0)\to (X,X\smallsetminus Y,x_0)$ and let
$\alpha$, $(K_j)$ and $(W_j)$ be as above.

Denote by $\Sigma :=\alpha^{-1}(Y)$ the \emph{singular set}. It is a closed subset
of $I^k$, thus is compact.
For all $s\in\Sigma$, there is a neighborhood
$U(s)$ of $\alpha(s)$ that is pathwise connected and such that $(U(s),U(s)\smallsetminus Y)$ is
$m$-connected. When $s\in K_j$, we can moreover assume that $U(s)\subset W_j$.
Let  $V(s)=\alpha^{-1}(U(s))$
and $B(s)$ be an open box around $s$ such that $\overline{B}(s)\subset V(s)$.
If $s\notin K_j$, we moreover assume that $\overline{B}(s)\cap K_j=\varnothing$.

Since $\Sigma$ is compact, there exist a finite number of points
$s_1,\ldots, s_N\in \Sigma$ such that the $B_i:=B(s_i)$ cover $\Sigma$.
Let us prove that $\alpha$ is homotopic to a map $\alpha_1$ for which $N$ can be reduced.
This will prove the theorem by induction on $N$, since $N=0$ means that $\alpha$
avoids $Y$.

Up to a reordering, we can assume that $B_1$ has its lower corner $x_1$
outside all of the $B_i$. In particular, $x_1\notin\Sigma$.
The restriction $\beta$ of $\alpha$ to the boundary of $B_1$ defines
a element in $\pi_{k-1}(U_1,x_1)$ where $U_1:=U(s_1)$. Since $(U_1,U_1\smallsetminus Y)$
is $m$-connected, $\beta$ is homotopic to a map
$\beta_1:\partial B_1\to U_1\smallsetminus Y$. The induction hypothesis moreover 
enables us to assume that the homotopy $(\beta_t)$
is small enough to ensure $\beta_t(K_j\cap\partial B_1)\subset W_j$
and 
$$\beta_t(\partial B_1\smallsetminus \bigcup_{i=2}^N B_i)\subset X\smallsetminus Y$$
for all $t$.
In particular, the homotopy $\beta_t$ will not add any new singular
part at the next step.

\begin{figure}[htp]\begin{center}
\input{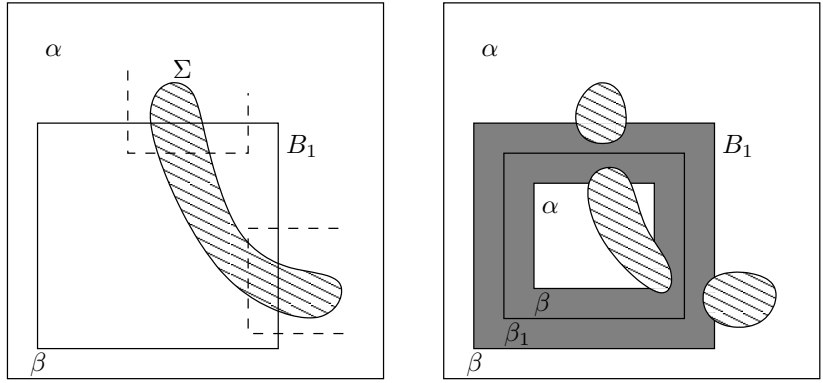}
\caption{On the left, dashed lines represent boxes $B_i$ for $i>1$.
  On the right, in the gray region we composed $\alpha$ with the homotopy from
  $\beta$ to $\beta_1$ and back to $\beta$.}\label{fig:composition}
\end{center}\end{figure}
Composing 
as in figure \ref{fig:composition} the part of $\alpha$ exterior to $B_1$
with the homotopy from $\beta$ to $\beta'$, then with its inverse,
finally with the restriction of $\alpha$ to $\overline{B}_1$, we can assume that $\alpha$
maps $\partial B_1$ in $U_1\smallsetminus Y$ while ensuring that $\Sigma$ is
still covered by the $B_i$.

Now the restriction of $\alpha$ to $\overline{B}_1$ defines an element
of $\pi_k(U_1,U_1\smallsetminus Y,x_1)$ and there is a homotopy
$H:[0,1]\times(B_1,\partial B_1,x_1)\to (U_1,U_1\smallsetminus Y,\alpha(x_1))$
such that $H(0,\cdot)=\alpha_{|\overline{B}_1}$ and $H(1,\cdot)$
takes its values in $U_1\smallsetminus Y$. This homotopy extends to a homotopy between $\alpha$
and a map $\alpha_1$ whose
singular set $\Sigma_1$ is covered by the $N-1$ boxes
$B_2,B_3,\ldots,B_N$. Moreover, our assumptions ensure that for all $j$,
either $B_1\cap K_j=\varnothing$ or $U_1\subset W_j$, therefore our homotopy
is small enough to carry the induction. This finishes the proof.

\subsection{Simple connectedness}

We note $\chab=\chab(\bR^n)$ and recall that $\mathscr{R}_m$ is the set of closed subgroups of maximal rank 
and $\mathscr{R}_\ell$ is its complement.

\begin{lemm}\label{lemm:lrc}
For each $\Gamma\in \mathscr{R}_\ell$, if we denote by $(U_\varepsilon)$ the neighborhood system
given in Corollary \ref{coro:system}, the pair
$$(U_\varepsilon,U_\varepsilon\smallsetminus \mathscr{R}_\ell)$$
 is $1$-connected.
\end{lemm}

\begin{proof}
We know that 
$U_\varepsilon$ is contractible, thus pathwise connected
and simply connected. The pair $(U_\varepsilon,U_\varepsilon\smallsetminus \mathscr{R}_\ell)$ is 
in particular $0$-connected.
Moreover $U_\varepsilon\smallsetminus \mathscr{R}_\ell$ is pathwise connected.
But we have an exact sequence 
$$1=\pi_1(U_\varepsilon)\to\pi_1(U_\varepsilon,U_\varepsilon\smallsetminus \mathscr{R}_\ell)
\to\pi_0(U_\varepsilon\smallsetminus \mathscr{R}_\ell)=1$$
thus $\pi_1(U_\varepsilon,U_\varepsilon\smallsetminus \mathscr{R}_\ell)$ is trivial, as desired.

This classical exact sequence is very easy to understand it this case:
any curve in $U_\varepsilon$ whose ends lie in $U_\varepsilon\smallsetminus \mathscr{R}_\ell$ is homotopic
to a curve whose ends coincide and lie in $U_\varepsilon\smallsetminus \mathscr{R}_\ell$, simply because
this set is arc-connected. But since $U_\varepsilon$ is simply connected,
this curve is nullhomotopic, thus homotopic to a curve entirely lying in $U_\varepsilon\smallsetminus \mathscr{R}_\ell$.
\end{proof}

We can now complete the proof of Theorem \ref{theo:main}.
Since $\mathscr{R}_\ell$ is the closure of the strata of type $(0,n-1)$ and thanks to the preceding lemma,
the localization theorem implies that $(\chab,\mathscr{R}_m)$ is simply connected. 
This means that any loop of $\chab$ based at $\bR^n$ is homotopic
to a loop in $\mathscr{R}_m$.

The map defined on $\mathscr{R}_m\times [0,1]$ by $H(\Gamma,t)=t\Gamma$ is a continuous homotopy between
the constant map with value $\bR^n$ and the identity map. Therefore, any loop of $\chab$ is nullhomotopic.
Note that the extension of $H$ on the whole of $\chab$
would not be continuous at $t=0$, since it fixes $0$ but retracts lattices of arbitrarily large norm
to $\bR^n$.

We cannot prove this way that $\chab$ is $2$-connected.
We would indeed need the $2$-connectedness of $(U_\varepsilon,U_\varepsilon\smallsetminus \mathscr{R}_\ell)$
which does not hold. For example, a typical neighborhood for $\bZ^{n-1}$ in $\bR^n$
has the homotopy type of the cone over a $(n-1)$-torus, its intersection with $\mathscr{R}_\ell$ being the apex.
The torus is not simply connected, thus the pair $(cT^{n-1},T^{n-1})$ is not $2$-connected.

\section{The Chabauty space of $\bR^2$ is a $4$-sphere}\label{sec:plane}

\subsection{Definitions and notations}

In this section, we denote by $\chab$ the Chabauty space of $\bR^2$.
A closed subgroup of $\bR^2$ is of one of the following types:
\begin{itemize}
\item $(0,0)$: the trivial subgroup $0$ ;
\item $(0,1)$: isomorphic to $\bZ$ ;
\item $(0,2)$: isomorphic to $\bZ^2$ (these are the lattices) ;
\item $(1,0)$: isomorphic to $\bR$ ;
\item $(1,1)$: isomorphic to $\bR\times\bZ$ ;
\item $(2,0)$: the total group $\bR^2$.
\end{itemize}
Each type is an orbit of the action of $\GL(2;\bR)$ on $\chab$.
The set of lattices is $\mathscr{L}:=\chab^{(0,2)}$, its complement is denoted by $\mathscr{H}$.

A closed subgroup $\Gamma$ of $\bR^2$ has a determinant, or covolume,
$\cov(\Gamma)$. If $\Gamma$ is a lattice, it is its usual
determinant, that is the determinant of any direct base of $\Gamma$.
It is $0$ if $\Gamma$ is isomorphic to
$\bR\times\bZ$ or $\bR^2$, and $\infty$ if $\Gamma$ is isomorphic to
$\bZ$ or $0$. In other words, it is the $2$-dimensional volume
of the quotient $\bR^2/\Gamma$.
By convention, $\cov(\Gamma)$ takes simultaneously
all values in $[0,\infty]$ if $\Gamma$ is isomorphic to $\bR$.
So defined, the levels of $\cov$ are closed in $\chab$. Outside the set
$\mathscr{R}:=\chab^{(1,0)}$ of such subgroups, $\cov$ is a continuous function.
Beware that here, the letter $\mathscr{R}$ refers to $\bR$ (and not to the rank).

Let $\mathscr{C}_{\geqslant 1}$, respectively $\mathscr{C}_{\leqslant1}$, be the subsets
of $\chab$ defined by $\cov\geqslant1$ and $\cov\leqslant1$. These set both
contain $\mathscr{R}$.
Let $\mathscr{H}_{\geqslant1}=\mathscr{H}\cap \mathscr{C}_{\geqslant1}$ be the set of subgroups
isomorphic to $\bR$, $\bZ$ or $0$, and $\mathscr{H}_{\leqslant1}=\mathscr{H}\cap C_{\leqslant1}$
be the set of subgroups isomorphic to $\bR$, $\bR\times\bZ$ or
$\bR^2$.

Let $\mathscr{L}_1$ be the set of covolume $1$ lattices, and
$\mathscr{C}_1$ its closure. Then $\mathscr{C}_1$ is the union of $\mathscr{L}_1$
and of the set $\mathscr{R}$.

We use the usual identification $\bR^2\simeq\bC$,
so that any subgroup isomorphic to $\bR$
can be written in the form $e^{i\theta}\bR$.

We also define as before the norm (or systol)
$$\norme(\Gamma)=N_1(\Gamma)=\inf\ensemble{|x|}{x\in\Gamma\setminus\{0\}}$$
It is a continuous functions taking its values in
$[0,\infty]$.
Let $\mathscr{C}^1$ be the set of norm $1$ subgroups of $\bR^2$. A point
of $\mathscr{C}^1$ is either isomorphic to $\bZ$, or a lattice.
We  denote by $\mathscr{Z}^1$ the set $\mathscr{C}^1\setminus \mathscr{L}$.

Figure \ref{fig:notations} sums up this notations.

\begin{figure}[htp]\begin{center}
\input{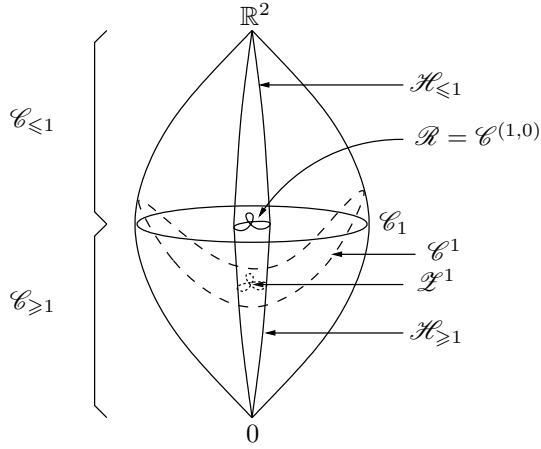}
\caption{Sum up of notations}\label{fig:notations}
\end{center}\end{figure}

The proof of Theorem \ref{theo:HP} is in two parts. We first prove that the topological pair
$(\mathscr{C},\mathscr{H})$ is the suspension of $(\mathscr{C}^1,\mathscr{Z}^1)$, then that the latter is homeomorphic
to $(S^3,K)$ where $K$ is a trefoil knot.

\subsection{The Chabauty space of $\bR^2$ is a suspension}

\begin{lemm}
The topological pair $(\mathscr{C}_{\geqslant1},\mathscr{H}_{\geqslant1})$ is homeomorphic
to the cone over $(\mathscr{C}_1,\mathscr{R})$.
\end{lemm}

\begin{proof}
We consider the map
\begin{eqnarray*}
\Phi : \mathscr{C}_1\times [0,\infty] & \to & \mathscr{C}_{\geqslant1} \\
       (\Gamma_1,t) &\mapsto& \left\{\begin{array}{l}
         \left(\frac{t}{\norme(\Gamma_1)}+1\right)\Gamma_1\mbox{ if }\Gamma_1\in L_1\\
         te^{i\theta} \bZ\mbox{ if } \Gamma_1=e^{i\theta}\bR 
       \end{array}\right.
\end{eqnarray*}
where by convention $0e^{i\theta}\bZ =  e^{i\theta}\bR$ and $\infty\Gamma=0$ 
if $\Gamma$ is discrete.

This map is continuous, maps $\mathscr{C}_1\times\{0\}$ onto $\mathscr{C}_1$ and $\mathscr{R}\times [0,\infty]$
onto $\mathscr{H}_{\geqslant1}$. It induces a continuous bijection $\tilde\Phi$ from the quotient of $\mathscr{C}_1\times [0,\infty]$
by the relation $(\Gamma_1,\infty)\sim(\Gamma'_1,\infty)$ onto $\mathscr{C}_{\geqslant1}$. Since
the latter is compact, $\tilde\Phi$ is a homeomorphism between the cone over
$(\mathscr{C}_1,\mathscr{R})$ and $(\mathscr{C}_{\geqslant1},\mathscr{H}_{\geqslant1})$.
\end{proof}

\begin{lemm}\label{lemm:norme-covol}
The topological pair $(\mathscr{C}_1,\mathscr{R})$ is homeomorphic to $(\mathscr{C}^1,\mathscr{Z}^1)$.
\end{lemm}

\begin{proof}
The map $\Psi: \mathscr{C}^1\to \mathscr{C}_1$ that assigns to $\Gamma$ the only $t\Gamma$
of unit covolume ($t=0$ if $\Gamma$ is isomorphic to $\bZ$,
$t=\cov(\Gamma)^{-1/2}$ otherwise) is continuous and a bijection.
 By compacity of $\mathscr{C}_1$, closed in $\mathscr{C}$, it is a homeomorphism.
\end{proof}

\begin{prop}
The topological pair $(\mathscr{C},\mathscr{H})$ is homeomorphic to the suspension
of $(\mathscr{C}^1,\mathscr{Z}^1)$.
\end{prop}

\begin{proof}
We can either reproduce the previous arguments
to prove that $(\mathscr{C}_{\leqslant1},\mathscr{H}_{\leqslant1})$ is also a cone
over $(\mathscr{C}^1,\mathscr{Z}^1)$ or use the duality $\ast$
which maps $\mathscr{C}_{\geqslant1}$ on $\mathscr{C}_{\leqslant1}$ and
preserves $\mathscr{L}$.
\end{proof}

\subsection{Subgroups of unit norm}

To get Theorem \ref{theo:HP}, we have left to prove the following.

\begin{prop}\label{prop:seifert}
The topological pair $(\mathscr{C}^1,\mathscr{Z}^1)$ is homeomorphic to $(S^3,K)$.
\end{prop}

The proof runs over the rest of the Section.
We shall describe $\mathscr{C}^1$ as a Seifert fibration. Let $\Gamma$
be a point of $\mathscr{C}^1$. The isometry group $\SO(2)$ acts on $\mathscr{C}^1$,
and up to a rotation we can assume that $1\in\Gamma\subset\bC$.
Then $\Gamma$ is determined by the choice of a second vector
in the fundamental domain
$$D=\{z\in\bC;| z|\geqslant 1\mbox{ and }-1/2\geqslant \repart(z)\geqslant1/2\}\cup\{\infty\}$$
where $z=\infty$ means that $\Gamma$ is isomorphic to $\bZ$ (figure \ref{fig:domaine_fondamental}).
Identifying the points of $D$ that represent the same $\Gamma$
leads to the quotient of $D$ by the relation $z\sim z-1$ if $\repart(z)=1/2$ and
$z\sim -\bar z$ if $| z|=1$, turning it into a $2$-sphere denoted by $B$, that
will be the base of the Seifert fibration.

\begin{figure}[htp]\begin{center}
\includegraphics[scale=.7]{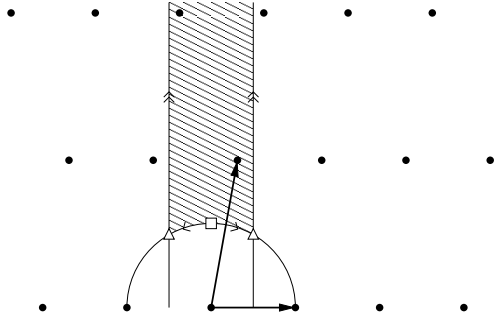}
\caption{Fundamental domain: the vertical lines and the circle arcs are glued according
         to the arrows, $\square$ and $\triangle$ are the singular points.}\label{fig:domaine_fondamental}
\end{center}\end{figure}

The kernel of the action of $\SO(2)$ is reduced to $\{\pm 1\}$, and the quotient
gives an action of the circle that is almost free: the only points of $\mathscr{C}^1$
that have nontrivial stabilizers are the triangular lattices (stabilizer
of order $3$) and the square lattices (stabilizer of order $2$). It follows
that $\mathscr{C}^1$ is a Seifert fibration with base $B\simeq S^2$ and two singular fibers
of order $2$ and $3$, and where $\mathscr{Z}^1$ is a regular fiber. The unnormalized Seifert invariants
of $\mathscr{C}^1$ are $(0|(2,\beta_1);(3,\beta_2))$ and we have left
to find the rational Euler number $\beta_1/2+\beta_2/3$ to
determine $(\mathscr{C}^1,\mathscr{Z}^1)$.

We first choose a cross-section of the regular part of the Seifert fibration. It would be
natural to lift each point $u$ in the fundamental domain to the subgroup generated
by $u$ and $1$, but this would not define a continuous cross-section. The
gluing of the unit circle indeed identifies, for all $\theta\in[0,\pi/6]$,
the subgroups $1\bZ+e^{i(\pi/2-\theta)}\bZ$ and $1\bZ+e^{i(\pi/2+\theta)}\bZ$ by a rotation
of angle $\pi/2+\theta$. We shall therefore modify this cross-section in a neighborhood
of one of the circular arcs of $D$.

Let $S^1=\bR/\pi\bZ$ be the quotient  $\SO(2)/\{\pm1\}$, $D'$ be the
fundamental domain $D$ minus the singular points ($i$, $e^{i\pi/3}$ and $e^{2i\pi/3}$)
and $B'$ be the base $B$ minus the two singular points (corresponding to $i$ and
$e^{i\pi/3}\sim e^{2i\pi/3}$). We choose a continuous map $f:D'\to [0,\pi/2]$
that is constant with value $0$ except in a neighborhood of the arc 
$\ensemble{e^{i(\pi/2+\theta)}}{\theta\in]0,\pi/6[}$,
where it satisfies $f(e^{i(\pi/2+\theta)})=\pi/2-\theta$. We then define a cross-section
$\sigma:B'\to \mathscr{C}^1$ by $\sigma(u)=e^{if(u)}(1\bZ+u\bZ)$. It is continuous
since
$$\sigma(e^{i(\pi/2+\theta)})=e^{i(\pi/2-\theta)}(1\bZ+e^{i(\pi/2+\theta)}\bZ)
                             =e^{i(\pi/2-\theta)}\bZ+1\bZ
                             =\sigma(e^{i(\pi/2-\theta)})$$

Let $b$ be the homotopy class in $\mathscr{C}^1$ of a regular fiber, $d_1$ and $d_2$ be the homotopy
classes defined by $\sigma$ on the boundary of $\mathscr{C}^1_\bullet=\mathscr{C}^1\setminus\{F_1,F_2\}$
where $F_1$ and $F_2$ are invariant neighborhoods of the singular fibers of order
$2$ and $3$, respectively.
\begin{figure}[htp]\begin{center}
\input{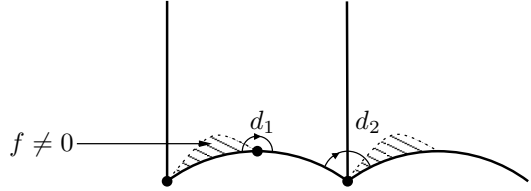}
\caption{The cross-section $\sigma$ defines homotopy classes in the boundary of $\mathscr{C}^1_\bullet$.}
\end{center}\end{figure}

In $\partial F_1$ and $\partial F_2$ respectively, we get that $2d_1+b$ and $3d_2-b$ are homotopic to meridians
(see figure \ref{fig:recol} where $F_1$ and $F_2$ are pictured with coordinates
$(u,\varphi)\in B\times \bR/\pi\bZ \mapsto e^{i\varphi}(1\bZ+u\bZ)$, with the suitable
identifications). It follows that $\mathscr{C}^1$ has unnormalized Seifert invariants
$(0|(2,1),(3,-1))$ and rational Euler number equal to $1/2-1/3=1/6$.

\begin{figure}[htp]
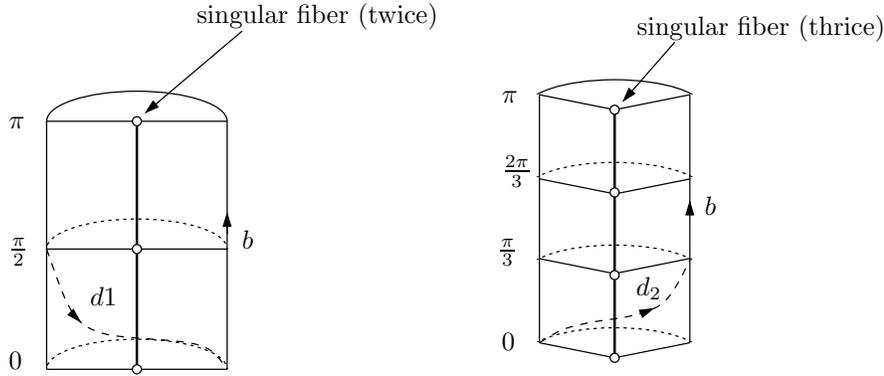
\begin{center}
\parbox{.45\textwidth}{\input{recol1.pstex_t}}\hspace{.05\textwidth}\parbox{.45\textwidth}{\input{recol2.pstex_t}}
\caption{Neighborhood $F_1$ and $F_2$ of the singular fibers}\label{fig:recol}
\end{center}\end{figure}

We shall know exhibit a very classical Seifert fibration on  $S^3$ whose regular fibers
are trefoil knots, that has base $S^2$, two singular fibers of order $2$ and $3$ and
rational Euler number $1/6$.
Since a Seifert fibration is determined by these data, we will conclude that $(\mathscr{C}^1,\mathscr{Z}^1)$ is
homeomorphic to $(S^3,K)$.

Consider the following action of the circle $\bR/\bZ$
on $S^3$, identified to the unit sphere of $\bC^2$:
$$s\cdot(z_1,z_2)=(e^{2\pi m_1is} z_1,e^{2\pi m_2is} z_2)$$
with $m_1=2$ and $m_2=3$.
The stabilizer of almost every point is trivial, the exceptions being
the polar orbits $(z_1,0)$ and $(0,z_2)$.
If $m_1$ and $m_2$ where equal to $1$, we would get the Hopf
fibration where the non-polar orbits are Villarceau circles
of the tori $|z_1/z_2|=c$, where $c$ runs over $[0,\infty]$.
Taking $m_1=2$ and $m_2=3$, we replaced the Villarceau circle by toric
knots, here trefoil knots (figure \ref{fig:trefle}).

\begin{figure}[htp]\begin{center}
\includegraphics[scale=.2]{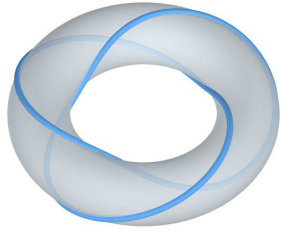}
\caption{The torus knot $(2,3)$ is a trefoil knot.}\label{fig:trefle}
\end{center}\end{figure}

 We see that
the regular part of the base is foliated by the circles obtained
by quotienting the tori $|z_1/z_2|=c$ by the action of $S^1$, and is therefore
an annulus. One can see this annulus as the $S^2$ base of the Hopf fibration
minus two points for the singular fibers.

Let us compute the Seifert invariants of this action, which are surprisingly
difficult to find in the litterature. We use a representation found in
\cite{Massot}.

Let $T^2=\bR/\bZ\times \bR\bZ$ be the standard $2$-torus
equiped with the foliation by straight lines of slope $3/2$. If we denote
by $x$ the homotopy class of 
\begin{eqnarray*}
\bR/\bZ  &\to&   T^2 \\
t        &\mapsto& (t,0)
\end{eqnarray*}
and by $y$ the homotopy class of
\begin{eqnarray*}
\bR/\bZ  &\to&   T^2 \\
t        &\mapsto& (0,t)
\end{eqnarray*}
the homotopy class of any leave of this foliation is $\ell=2x+3y$.

In the space $T^2\times[0,1]$ define $T_t:=\mathbb{T}^2\times\{t\}$,
endowed with the above foliation for $t\in (0,1)$.
Let $\Pi:T^2\times[0,1]\to S^3$ be the mapping defined as follows.
First, $\Pi$ contracts $T_0$ to the singular fiber $\ensemble{(0,z_2)}{|z_2|=1}$
and $T_1$ to the singular fiber $\ensemble{(z_1,0)}{|z_1|=1}$
with $\Pi(a,b,0)=(0,e^{2i\pi b})$ and $\Pi(a,b,1)=(e^{2i\pi a},0)$.
Second, it maps $T_t$ to a torus defined by $|z_1/z_2|=c(t)$
with $c$ an increasing continuous function such that  $c(t)\to 0$ (resp. $+\infty$) when
$t\to 0$ (resp. $1$), and maps the foliation of $T_t$ to the Seifert
foliation in $S^3$. Think of $T^2\times[0,1]$ as a blow-up
of $S^3$ along the singular fibers.

The point is that in this presentation, one can give explicitely a cross-section
of the Seifert fibration over the regular part: just consider the set
$$\ensemble{(s,2s,t)}{s\in\bR/\bZ, t\in(0,1)}\subset T^2\times(0,1)$$ 
This set intersects each of the $T_t$ along a straight line
homotopic to $x+2y$, which intersects each $2x+3y$ line once, thus it does define a section.

In the boundary of a neighborhood of $T_0$, the section defines a curve
homotopic to $d_0=-x-2y$ (the sign depends upon the choice of orientation).
 Since $\ell=2x+3y$ is
the homotopy class of a regular fiber, we have
$3d_0+2\ell=x$, a meridian.
Similarly, in the boundary of a neighborhood of $T_1$,
the section defines a curve homotopic to $d_1=x+2y$ 
and $2d_1-\ell=y$ is a meridian.

Therefore, this Seifert fibration has unnormalized invariants
$(0|(3,2),(2,-1))$ and rational Euler number $2/3-1/2=1/6$ as needed.

\begin{rema}
It is well known (see e.g. \cite{Rubinstein-Gardiner}) that
$\SL(2;\bR)/\SL(2;\bZ)$ is homeomorphic to the complement
of a trefoil knot in $S^3$. This can be given an alternative
proof using the same methods we used to prove Theorem \ref{theo:HP}.
Lemma \ref{lemm:norme-covol} indeed shows that this homogeneous
space is homeomorphic to $\mathscr{C}^1\setminus \mathscr{Z}^1$. The study of the Seifert
fibration on $\mathscr{C}^1$ implies this result,
but the study of $\mathscr{C}^1\setminus \mathscr{Z}^1$ is in fact simpler. More precisely,
we do not need to study how the singular fibers are glued to the regular
part, the non-compacity enabling one to assume $\beta_1=1$ and
$\beta_2=1$ or $2$. Moreover, the difference between $\beta_2=1$
and $\beta_2=2$ $(\equiv -1\mod3)$ lies only in the orientation, and the result holds
regardless (but then we do not know if, for a given orientation of $\SL(2;\bR)/\SL(2;\bZ)$,
we obtain the complement of a right or left trefoil knot).
\end{rema}

\begin{rema}
Christopher Tuffley studied \cite{Tuffley} the spaces $\exp_k(S^1)$ of all 
non-empty subset of the circle of cardinality at most $k$. In particular,
he proved using Seifert fibrations that $\exp_3(S^1)$ is a $3$-sphere,
its subset $\exp_1(S^1)$ being a trefoil knot.

The similarity with Proposition \ref{prop:seifert} is not fortuitous:
Jacob Mostovoy proved \cite{Mostovoy} by a simple geometric argument  that $(\exp_3(S^1),\exp_1(S^1))$ is homeomorphic
to $(\mathscr{C}^1,\mathscr{Z}^1)$. Combining these two results one gets another Seifert fibration
proof of Proposition \ref{prop:seifert}. Note that even the Seifert part is somewhat different from ours,
since it is first proved that $\exp_3(S^1)$ is simply connected, which reduces drastically
its possible Euler numbers.
\end{rema}

\begin{rema}
A nice feature of the study of $\exp_3(S^1)$ is that its subset $\exp_2(S^1)$ is
easily seen to be a M{\"o}bius strip, with boundary  $\exp_1(S^1)$: we recover
the fact that a trefoil knot bounds a M{\"o}bius strip.
This can be seen in $(\mathscr{C}^1,\mathscr{Z}^1)$ as well: over the vertical line $L=\ensemble{iy}{y\in [1,+\infty]}$
of the base $B$, the Seifert fibration is a closed M{\"o}bius strip whith boundary $\mathscr{Z}^1$,
obtained by identifying antipodal points of the $(y=1)$ boundary component of the strip $L\times S^1$.
\end{rema}

\section{A few open questions}\label{sec:open}

There are many question left open concerning $\chab(\bR^n)$. Let us consider some of them
that seem of special interest.

\begin{enumerate}
\item Determine whether $\chab(\bR^n)$ is stratified in the sense of Thom or Mather.
It would for example imply that it can be triangulated (\cite{Johnson,Goresky}).
More ambitiously, determine if we can endow $\chab(\bR^n)$ with the structure of an algebraic variety.
This question is motivated by the original proof of the Hubbard-Pourezza theorem,
where the link $L^2(0)$ is described by algebraic means.
\item Compute the intersection homology of $\chab(\bR^n)$.
\item Describe explicitely $\chab(\bR^3)$, or at least the set $L^3(0)$ of unit norm subgroups
 of $\bR^3$.
\end{enumerate}

\bibliographystyle{smfplain}
\bibliography{biblio}

\providecommand{\bysame}{\leavevmode ---\ }
\providecommand{\og}{``}
\providecommand{\fg}{''}
\providecommand{\smfandname}{et}
\providecommand{\smfedsname}{\'eds.}
\providecommand{\smfedname}{\'ed.}
\providecommand{\smfmastersthesisname}{M\'emoire}
\providecommand{\smfphdthesisname}{Th\`ese}
\begin{thebibliography}{10}

\bibitem{Borel-Ji}
{\scshape A.~Borel {\normalfont \smfandname} L.~Ji} -- \emph{Compactifications
  of symmetric and locally symmetric spaces}, Mathematics: Theory \&
  Applications, Birkh\"auser Boston Inc., Boston, MA, 2006.

\bibitem{BHK}
{\scshape M.~R. Bridson, P.~de~la Harpe {\normalfont \smfandname} V.~Kleptsyn}
  -- {\og The chabauty space of closed subgroups of the three-dimensional
  heisenberg group\fg}, 2007, arXiv:0711.3736, to appear in \textit{Pacific J.
  Math.}

\bibitem{Chabauty}
{\scshape C.~Chabauty} -- {\og Limite d'ensembles et g\'eom\'etrie des
  nombres\fg}, \emph{Bull. Soc. Math. France} \textbf{78} (1950), p.~143--151.

\bibitem{Eyral}
{\scshape C.~Eyral} -- {\og Profondeur homotopique et conjecture de
  {G}rothendieck\fg}, \emph{Ann. Sci. \'Ecole Norm. Sup. (4)} \textbf{33}
  (2000), no.~6, p.~823--836.

\bibitem{Goresky-MacPherson1}
{\scshape M.~Goresky {\normalfont \smfandname} R.~MacPherson} -- {\og
  Intersection homology theory\fg}, \emph{Topology} \textbf{19} (1980), no.~2,
  p.~135--162.

\bibitem{Goresky-MacPherson2}
\bysame , {\og Intersection homology. {II}\fg}, \emph{Invent. Math.}
  \textbf{72} (1983), no.~1, p.~77--129.

\bibitem{Goresky-MacPherson3}
\bysame , \emph{Stratified {M}orse theory}, Ergebnisse der Mathematik und ihrer
  Grenzgebiete (3) [Results in Mathematics and Related Areas (3)], vol.~14,
  Springer-Verlag, Berlin, 1988.

\bibitem{Goresky}
{\scshape R.~M. Goresky} -- {\og Triangulation of stratified objects\fg},
  \emph{Proc. Amer. Math. Soc.} \textbf{72} (1978), no.~1, p.~193--200.

\bibitem{Guivarch-Remy}
{\scshape Y.~Guivarc'h {\normalfont \smfandname} B.~R{\'e}my} -- {\og
  Group-theoretic compactification of {B}ruhat-{T}its buildings\fg}, \emph{Ann.
  Sci. \'Ecole Norm. Sup. (4)} \textbf{39} (2006), no.~6, p.~871--920.

\bibitem{delaHarpe}
{\scshape P.~de~la Harpe} -- {\og Spaces of closed subgroups of locally compact
  groups\fg}, 2008, arXiv:0807.2030.

\bibitem{Johnson}
{\scshape F.~E.~A. Johnson} -- {\og On the triangulation of stratified sets and
  singular varieties\fg}, \emph{Trans. Amer. Math. Soc.} \textbf{275} (1983),
  no.~1, p.~333--343.

\bibitem{Massot}
{\scshape P.~Massot} -- {\og Geodesible contact structures on 3-manifolds\fg},
  \emph{Geom. Topol.} \textbf{12} (2008), no.~3, p.~1729--1776.

\bibitem{Mostovoy}
{\scshape J.~Mostovoy} -- {\og Lattices in {$\mathbb{C}$} and finite subsets of
  a circle\fg}, \emph{Amer. Math. Monthly} \textbf{111} (2004), no.~4,
  p.~357--360.

\bibitem{Hubbard-Pourezza}
{\scshape I.~Pourezza {\normalfont \smfandname} J.~Hubbard} -- {\og The space
  of closed subgroups of {${\bf R}\sp{2}$}\fg}, \emph{Topology} \textbf{18}
  (1979), no.~2, p.~143--146.

\bibitem{Rubinstein-Gardiner}
{\scshape J.~H. Rubinstein {\normalfont \smfandname} C.~Gardiner} -- {\og A
  note on a {$3$}-dimensional homogeneous space\fg}, \emph{Compositio Math.}
  \textbf{39} (1979), no.~3, p.~297--299.

\bibitem{Satake}
{\scshape I.~Satake} -- {\og On representations and compactifications of
  symmetric {R}iemannian spaces\fg}, \emph{Ann. of Math. (2)} \textbf{71}
  (1960), p.~77--110.

\bibitem{Siebenmann}
{\scshape L.~C. Siebenmann} -- {\og Deformation of homeomorphisms on stratified
  sets. {I}, {II}\fg}, \emph{Comment. Math. Helv.} \textbf{47} (1972),
  p.~123--136; ibid. 47 (1972), 137--163.

\bibitem{Thom}
{\scshape R.~Thom} -- {\og Ensembles et morphismes stratifi\'es\fg},
  \emph{Bull. Amer. Math. Soc.} \textbf{75} (1969), p.~240--284.

\bibitem{Tuffley}
{\scshape C.~Tuffley} -- {\og Finite subset spaces of {$S\sp 1$}\fg},
  \emph{Algebr. Geom. Topol.} \textbf{2} (2002), p.~1119--1145 (electronic).

\bibitem{Whitney1}
{\scshape H.~Whitney} -- {\og Complexes of manifolds\fg}, \emph{Proc. Nat.
  Acad. Sci. U. S. A.} \textbf{33} (1947), p.~10--11.

\bibitem{Whitney2}
\bysame , {\og Local properties of analytic varieties\fg}, Differential and
  {C}ombinatorial {T}opology ({A} {S}ymposium in {H}onor of {M}arston {M}orse),
  Princeton Univ. Press, Princeton, N. J., 1965, p.~205--244.

\bibitem{Whitney3}
\bysame , {\og Tangents to an analytic variety\fg}, \emph{Ann. of Math. (2)}
  \textbf{81} (1965), p.~496--549.

\end{thebibliography}

\end{document}